\documentclass{article}

\usepackage{arxiv}

\usepackage[utf8]{inputenc} 
\usepackage[T1]{fontenc}    
\usepackage{hyperref}       
\usepackage{url}            
\usepackage{booktabs}       
\usepackage{amsfonts}       
\usepackage{nicefrac}       
\usepackage{microtype}      
\usepackage{graphicx}
\usepackage{natbib}
\usepackage{doi}

\usepackage{amscd,amssymb,amsmath,
amsthm,mathrsfs,mathtext,enumerate,
float,mathtools}

\theoremstyle{plain}
\newtheorem{theorem}{Theorem}%
\newtheorem{proposition}[theorem]{Proposition}%
\newtheorem{lemma}{Lemma}%
\newtheorem*{corollary}{Corollary}%
\newtheorem{claim}{Claim}
\newtheorem*{claim*}{Claim}

\theoremstyle{plain}
\newtheorem{definition}{Definition}%
\newtheorem{conjecture}{Conjecture}%

\theoremstyle{plain}
\newtheorem*{hdefinition}{Heuristic definition of LBS}
\newtheorem*{htheorem}{Heuristic main theorem}

\theoremstyle{remark}
\newtheorem{remark}{Remark}%
\newtheorem{example}{Example}%

\newcommand{\ssubset}{\subset\joinrel\subset}

\newcommand{\AH}{\mathrm{AH}}
\newcommand{\SN}{\mathrm{SN}}
\newcommand{\HC}{\mathrm{HC}}
\newcommand{\SC}{\mathrm{SC}}
\newcommand{\SL}{\mathrm{SL}}
\newcommand{\PC}{\mathrm{PC}}

\DeclareMathOperator{\Vect}{Vect}

\newcommand{\supp}[1]{\operatorname{supp}{#1}}

\newcommand{\Acc}{\mathrm{Acc}}
\newcommand{\LBS}{\mathrm{LBS}}
\newcommand{\ELBS}{\mathrm{ELBS}}
\newcommand{\BD}{\mathrm{BD}}

\DeclareMathOperator{\stab}{\mathrm{Stab}}
\DeclareMathOperator{\splt}{\mathrm{Split}}

\makeatletter
\newcommand{\proofpart}[2]{%
  \par
  \addvspace{\medskipamount}%
  \noindent\emph{Part #1: #2}\par\nobreak
  \addvspace{\smallskipamount}%
  \@afterheading
}
\makeatother

\newcommand{\defgerm}{Given a point $x$ of a topological space $X$, and two maps $f,\ g: X \to Y$ (where $Y$ is any set), the maps $f$ and $g$ define the same germ at $x$ if there is a neighborhood $U$ of $x$ such that restricted to $U$, $f$ and $g$ are equal: $f(u)=g(u)$ for all $u$ in $U$. Similarly, if $S$ and $T$ are any two subsets of $X$, then they have the~same germ at $x$ if there is again a neighborhood $U$ of $x$ such that $S \cap U = T \cap U$.}

\newcommand{\defsep}{An orbit $\gamma$ of a vector field is a separatrix if there is a singular point of the vector field with a hyperbolic sector $S$ s.t. $\gamma$ contains a curve that belongs to the boundary of this sector.}

\newcommand{\sparkle}{Sparkling saddle connections is a phenomenon that occurs in a generic family that unfolds a vector field with a parabolic cycle when at least one separatrix of hyperbolic saddles is winding on the cycle from both sides. Vanishing of this cycle is accompanied by an infinite series of saddle connections between the separatrices. The corresponding orbits connecting the saddles accumulate to the cycle in the sense that the cycle belongs to the closure of these orbits in the total space of the family.}

\title{Disconnected large bifurcation supports and Cartesian products of bifurcations}
\author{ \href{https://orcid.org/0000-0002-9481-3816}{\hspace{1mm}Timur~Bakiev}\thanks{Alternative email: tnbakiev@edu.hse.ru} \\
	Faculty of Mathematics\\
	HSE University\\
	Moscow, Russia 109028 \\
	\texttt{tnbakiev@hse.ru} \\
	\And
	\href{https://orcid.org/0000-0003-1087-5903}{\hspace{1mm}Yulij S.~Ilyashenko} \\
	Faculty of Mathematics\\
	HSE University\\
	Moscow, Russia 109028 \\
}

\renewcommand{\shorttitle}{\textit{arXiv} Cartesian products of bifurcations}

\hypersetup{
pdftitle={Disconnected large bifurcation supports and Cartesian products of bifurcations},
pdfsubject={math.DS, math.CA},
pdfauthor={Timur~Bakiev, Yulij S.~Ilyashenko},
pdfkeywords={Planar vector fields, Bifurcations, Cartesian products},
}

\begin{document}
\maketitle

\begin{abstract}
	A bifurcation that occurs in a multiparameter family is a Cartesian product if it splits into two factors in the sense that one bifurcation takes place in one part of the phase portrait, another one -- in another part, and they are in a sense independent, do not interact with each other. To understand how a~family bifurcates, it is sufficient to study it in a neighborhood of~the~so-called large bifurcation support. Given a family of vector fields on $S^2$ that unfolds a field $v_0$, the respective large bifurcation support is a closed $v_0$-invariant subset of the sphere indicating parts of the phase portrait of $v_0$ affected by bifurcations. One should consider disconnected large bifurcation supports in order to obtain Cartesian products for sure. We prove that, if the large bifurcation support is disconnected and the restriction of the original family to some neighborhood of each connected component is structurally stable (plus some mild extra conditions), then the original family is a~Cartesian product of the bifurcations that occur near the components of the large bifurcation support. We also show that the structural stability requirement cannot be omitted.
\end{abstract}


\section{Introduction}
\label{sec:introduction}

This paper studies a special class of bifurcations on the two-sphere, so-called Cartesian products. Roughly speaking, a family is a Cartesian product of bifurcations if there are two parts of the phase space, one bifurcation occurs in the first part, another one in the second, and those bifurcations are in a sense independent. We give a sufficient condition for a family to be (equivalent to) a Cartesian product. It is stated in terms of large bifurcation supports and the extension triviality property.

In this introduction we give the heuristic version of the definitions and results (we hope they are easy to read) so that the reader will have an idea of what the paper is about. After that we present the rigorous version of these statements, which is more lengthy.

Large bifurcation support of a family is a subset of the phase space responsible for bifurcations that occur in the family. It was introduced in \citep{Goncharuk-LBS} by Goncharuk and Ilyashenko for glocal families

\begin{definition}[Glocal families]\label{gloc}
    Denote $\Vect^{*}_k\left( M \right)$ the subset of $\Vect^{k}\left( M \right)$, the Banach space of $C^k$-smooth vector fields on $M$, such that each vector field $v \in \Vect^{*}_k\left( M \right)$ has only a finite number of limit cycles and singular points counted with multiplicity\footnote{Multiplicity of a cycle (singular point) is the maximal number (if exists) of cycles (singular points) that a $C^k$-close fields has in an arbitrary small neighbourhood of the cycle (singular point). In particular, each $v \in \Vect^{*}_k\left( M \right)$ does not have zero $k$-jet at any point on $S^2$.}.

    A $k$-smooth $n$-parameter family ($n \in \mathbb{N}$) of vector fields on a manifold $M$ is a $C^k$-smooth map $B \to \Vect^{k}\left( M \right)$, where $B \subset \mathbb{R}^n$ is a domain. Given a~$k$-smooth family $V$ of vector fields on a manifold $M$ we call $M$ the~phase space of $V$, and $U$ -- the base of $V$. The total space of the family is a Cartesian product of the phase space and the base.

    A glocal (resp. $M$-glocal)\footnote{\textbf{glo}bal phase space + \textbf{local} base of parameters} $n$-parameter family ($n \in \mathbb{N}$) of vector fields is a~germ\footnote{\defgerm} of a~$C^k$-smooth map $\left(\mathbb{R}^n,\ 0\right) \to \Vect^{*}_k\left(S^2\right)$ (resp. $\left(\mathbb{R}^n,\ 0\right) \to \Vect^{*}_k\left(M\right)$)\footnote{It follows that the corresponding extended system of ordinary differential equations $\dot{x} = v_\varepsilon (x)$, $\dot{\varepsilon} = 0$ has $C^k$-smooth right hand side.}.
\end{definition}

\begin{remark}
    The base of a glocal or $M$-glocal $n$-parameter family is the germ $\left(\mathbb{R}^n,\ 0\right)$, while the phase space is $S^2$ or $M$ respectively. A representative of the base is any open neighborhood of the origin in $\mathbb{R}^n$. 
    
    A glocal ($M$-glocal) family of vector fields is a class of equivalence. A~representative of a glocal ($M$-glocal) $n$-parameter family of vector fields is a~$k$-smooth $n$-parameter family of vector fields on $S^2$ ($M$).
    
    When it comes to calculations or any other actions one needs to perform with a glocal or $M$-glocal family, one has to choose a representative. We want to save the reader from a dozen stereotypical phrases, so we will overload the~notation a little: we will use the same capital Latin letters for both the~glocal ($M$-glocal) family and its representative. We hope this does not lead to any significant confusion.
\end{remark}

\begin{remark}
    In what follows, the smoothness degree $k$ of families considered is at~least~$2$. 
\end{remark}

We will use many times a notion of the large bifurcation support ($LBS$) of a glocal family. Let us first present a heuristic definition of the large bifurcation support. The formal one will be given later (see Def.~\ref{def:cLBS}).

\begin{hdefinition}\label{def:heurLBS}
    Given a glocal family $V = \{v_\varepsilon\}$, define a set $Z = \LBS(V) \subset S^2$ as one characterized by the following properties:

    \begin{itemize}
        \item $Z$ is closed and $v_0$-invariant;
        \item if two glocal families $W = \{w_\varepsilon\}$ and $V$ are "equivalent" in some neighborhoods of their large bifurcation supports, $w_0$ is orbitally topologically equivalent to $v_0$, and both equivalences in a sense agree, then $V$ and $W$ are "equivalent" on the whole $S^2$.
    \end{itemize}
\end{hdefinition}

The term "equivalent" is put in quotation marks, because there are several ways to define equivalence of glocal families: there are weak, strong, sing, moderate equivalences and so on. The formal definition of large bifurcation supports demands to specify the kind of equivalence used.

Another key concept of this paper is \textit{the extension triviality property}. It is close to the notion of structural stability. Recall that a glocal family is structurally stable provided that it is equivalent to any nearby family (in~the~$C^k$-topology).

\begin{definition}
   A family $V = \{v_\varepsilon\}$ has the extension triviality property if any family $W = \{w_{(\varepsilon,\ \delta)}\}$ having $V$ as a subfamily: $w_{(\varepsilon,\ 0)} = v_\varepsilon$, is equivalent to a family $\widetilde{W} = \{\widetilde{w}_{(\varepsilon,\ \delta)}\}$, where $\widetilde{w}_{(\varepsilon,\ \delta)} = v_\varepsilon$ (i.e. vector fields in $\widetilde{W}$ do not depend on $\delta$). The latter family is called thereby a trivial extension of $V$. 
\end{definition}

Our main result may be heuristically stated as follows.

\begin{htheorem}
    Let a glocal family $V$ have a disconnected large bifurcation support $Z = Z_1 \sqcup Z_2$. Suppose that restrictions of $V$ to some neighborhoods $U_1 \supset Z_1$ and $U_2 \supset Z_2$ have the extension triviality property. Then $V$ is equivalent to a Cartesian product of bifurcations.
\end{htheorem}

There are several upcoming results concerning structural stability of one-parameter families (\citep{Androsov}) and properties of generic two-parameter families (e.g.~\citep{Filimonov-24}) that align with this theorem to constitute a rigorous proof of the~following hypothesis:

\begin{conjecture}\label{th2}
    A generic two-parameter family with a disconnected large bifurcation support is equivalent to a Cartesian product of bifurcations.
\end{conjecture}

In the same manner we hope to prove the next two conjectures.

\begin{conjecture}
    Generic $n$-parameter family with an $\LBS$ having $n$ connected components is equivalent to a Cartesian product of $n$ one-parameter bifurcations.
\end{conjecture}

\begin{conjecture}
    Generic $3$-parameter family with a disconnected $\LBS$ is equivalent to a Cartesian product either of $3$ one-parameter bifurcations, or of $2$ bifurcations, one of them $1$-parameter and another one $2$-parameter.
\end{conjecture}

\begin{remark}
    We believe that an analogous statement is no longer true for generic $4$-parameter families. There is a serious obstacle due to presence of numerical invariants. We know that the topological type of some non-local bifurcations of codimension $3$ is determined by intrinsic numerical characteristics of the~degeneracy~\citep{Ilyashenko2018,Dukov2020,Goncharuk_2021}, which, in turn, may depend non-trivially on parameters in a family.
\end{remark}

\section{Definitions and results}
\label{sec:definitions}

Here we define all terms needed for the formal statement of our main theorem. The definitions use notions of weak and moderate equivalences of~glocal families that we recall in the next subsection.

Let us introduce a notation. Given a family $V = \{v_\varepsilon\}$ of vector fields on~$M$ and a subset $U$ of $M$, we denote $V|_U := \{v_\varepsilon|_U\}$ a family obtained from~$V$ by restricting each field in $V$ to $U$.

We use letters $B$, $\widetilde{B}$, $B_1$, $B_2$, etc. to denote either an open neighborhood of the origin in $\mathbb{R}^n$, homeomorphic to the unit disk, or a germ of subsets of real finite-dimensional Euclidean space at zero, i.e. $(\mathbb{R}^n,\ 0)$ for some $n \in \mathbb{N}$. Sometimes we use the same letter both for a germ and its representative. The~right interpretation is usually clear from the context.

\subsection{Moderate equivalence}
\label{subsec:moderate}

This paper is based on notions of the large bifurcation support and the~moderate equivalence introduced in \citep{Goncharuk-LBS}. Let us review first two classical ways to define equivalent families.

\begin{definition}[Weak and strong equivalences]\label{weak&strong}
    Two families of vector fields $V = \{v_\varepsilon \left|\ \varepsilon \in B \right.\}$ and $\widetilde{V} = \{ \left. \widetilde{v}_{\widetilde{\varepsilon}}\ \right| \ \widetilde{\varepsilon} \in \widetilde{B}\}$ on $M$ are called weakly equivalent if there is a map
    $$
        \mathbf{H} : B \times M \to \widetilde{B} \times M, \quad \mathbf{H}(\varepsilon,\ x) = \left(h(\varepsilon),\ H_\varepsilon (x)\right),
    $$
    such that $h : B \to \widetilde{B}$ is a homeomorphism between the bases, and for each $\varepsilon \in B$ $H_\varepsilon : M \to M$ is an orbital topological equivalence between $v_\varepsilon$ and $\widetilde{v}_{h(\varepsilon)}$.

    If, in addition, $\mathbf{H}$ is a homeomorphism itself, then $V$ and $\widetilde{V}$ are called strongly equivalent.
\end{definition}

Let $V = \{v_\varepsilon \left|\ \varepsilon \in B \right.\}$ be a family of vector fields on $S^2$. Denote by $\Acc{(V)}$ the~set
\begin{equation}\label{eqn:acc}
    \Acc{(V)} = (\overline{\operatorname{Per} V} \cup \overline{\operatorname{Sep} V}) \cap\{\varepsilon=0\}
\end{equation}
where $\operatorname{Per} V$ and $\operatorname{Sep} V$ live in the product $B \times S^2$ and consist of all points belonging to periodic trajectories or separatrices\footnote{\defsep} of $V$ respectively. Let
$$
\begin{aligned}
    C_\partial(V) & := \operatorname{S} (v_{0}) \cup \partial \Acc{(V)},
\end{aligned}
$$
where $\operatorname{S} (v_{0}) := \operatorname{Sing} (v_0) \cup \operatorname{Per} (v_0) \cup \operatorname{Sep} (v_0)$\footnote{We denote by $\operatorname{Sing}(v_0)$ the set of all singular points of $v_0$; $\operatorname{Per} (v_0)$ and $\operatorname{Sep} (v_0)$ are the unions of all limit cycles and separatrices of $v_0$ respectively.}. This union is called the~separatrix skeleton of $v_0$.

\begin{definition}[Moderate equivalence]\label{moderate}
    Two families of vector fields on $S^2$, $V = \{v_\varepsilon \left|\ \varepsilon \in B \right.\}$ and $\widetilde{V} = \{ \left. \widetilde{v}_{\widetilde{\varepsilon}}\ \right| \ \widetilde{\varepsilon} \in \widetilde{B}\}$, are moderately equivalent in some neighborhoods of the closed sets $Z_{1},\ Z_{2} \subset S^2$ if
    \begin{enumerate}
        \item $Z_{1}$ is $v_{0}$-invariant, and $Z_{2}$ is $\widetilde{v}_{0}$-invariant;
        \item there exists a neighborhood $U \supset Z_{1}$ and a map
        $$
        \begin{aligned}
            \mathbf{H} : B \times U & \rightarrow \widetilde{B} \times S^{2} \\
            (\varepsilon,\ x) & \mapsto \left(h(\varepsilon),\ H_{\varepsilon}(x)\right)
        \end{aligned}
        $$
        such that
        \begin{itemize}
            \item[i.] $h$ is a homeomorphism between the bases;
            \item[ii.] $\forall \varepsilon \in B$ the map $H_{\varepsilon}: U \rightarrow S^{2}$ is an orbital topological equivalence between $\left.\left(v_{\varepsilon}\right)\right|_{U}$ and $\left.\left(\widetilde{v}_{h(\varepsilon)}\right)\right|_{H_{\varepsilon}(U)}$;
        \end{itemize}
        \item the map $\mathbf{H}$ is continuous with respect to $(\varepsilon,\ x)$ on the intersection of its domain with $C_\partial(V)$; the map $\mathbf{H}^{-1}$ is continuous with respect to $(\widetilde{\varepsilon},\ x)$ on the intersection of its domain with $C_\partial(\widetilde{V})$;
        \item $H_{0}\left(Z_{1}\right)=Z_{2}$, and for each neighborhood $G$ of $\{\varepsilon=0\} \times Z_{1}$, its image $\mathbf{H}(G)$ contains some neighborhood of $\{\widetilde{\varepsilon}=0\} \times Z_{2}$, and the same holds for the inverse map $\mathbf{H}^{-1}$.
    \end{enumerate}
\end{definition}

\begin{definition}
    We say that two glocal ($M$-glocal) families are weakly/strongly/moderately equivalent, if they have weakly/strongly/moderately equivalent representatives.
\end{definition}

Discussion of various types of~equivalences between glocal families may be found in \citep{Goncharuck}. The main idea is that the strong equivalence is too strong: it distinguishes the families that are equivalent from the intuitive point of view, while the weak equivalence is too weak: it identifies the families that are different from the intuitive point of view. The moderate equivalence is more adequate, but also more sophisticated.

\subsection{LBS}
\label{subsec:lbs}

We can give a precise version of the heuristic Definition \ref{def:heurLBS} now.

\begin{definition}[LBS]\label{def:axLBS}
    Given a glocal family $V = \{v_\varepsilon \left|\ \varepsilon \in B \right.\}$, define a set $Z = \LBS(V) \subset S^2$ as one characterized by the following properties:
    \begin{itemize}
        \item $Z$ is closed and $v_0$-invariant;
        \item if a glocal family $W = \{w_\varepsilon \left|\ \varepsilon \in B \right.\}$ and $V$ are moderately equivalent in some neighborhoods of their large bifurcation supports, $w_0$ is orbitally topologically equivalent to $v_0$, and the moderate equivalence agrees with the topological equivalence for $\varepsilon = 0$, then $V$ and $W$ are weakly equivalent on the whole sphere.
    \end{itemize}
\end{definition}

This is a so-called axiomatic definition. Of course, the whole phase space, for example, satisfies this definition. But interesting are smaller sets. We~give a constructive definition of $\LBS$ now.

\begin{definition}[Interesting limit sets]\label{def:int}
    A nest of limit cycles of a vector field is the maximal set of nested cycles with no singular points in between them.

    A limit cycle is interesting if its nest contains only cycles of even multiplicity, and both domains on the sphere bounded by this cycle contain at least one singular point different from a hyperbolic attractor or repeller.

    An $\alpha$- or $\omega$-limit set of a point $x$ is called interesting if it contains a~singular point different from a hyperbolic attractor or repeller, or coincides with an interesting limit cycle.
\end{definition}
    
\begin{definition}[ELBS]\label{def:elbs}
    Extra large bifurcation support of a vector field $v_0$ on $S^2$ is the union of all non-hyperbolic singular points and limit cycles of this field, plus the closure of the set of all points for which both $\alpha$- and $\omega$-limit sets are interesting. It is denoted $\ELBS (v_0)$ and called the extra large bifurcation support of $v_0$.
\end{definition}

It is clear that $\ELBS (v_0)$ is $v_0$-invariant.

\begin{definition}[LBS]\label{def:cLBS}
    \begin{equation}\label{eqn:lbs}
        \begin{aligned}
            \LBS (V) &= \ELBS (v_0) \cap (\operatorname{Sing} (v_{0}) \cup \operatorname{Acc}{(V)}),
        \end{aligned}
    \end{equation}
    see~\ref{eqn:acc} for the definition of $\operatorname{Acc}$.
\end{definition}

Actually, we have described a map from glocal families to subsets~of~$S^2$. For each glocal family it yields some relatively small set that meets all the~requirements listed in the axiomatic definition.

\begin{theorem}[\citep{Goncharuk-LBS}]\label{LBS-main-theorem}
    The $\LBS$ thus defined satisfies the axiomatic Definition \ref{def:axLBS}.
\end{theorem}

This is a difficult result with a long proof. It allows us to determine explicitly the $\LBS$ in many cases and simplifies the study of unfoldings of~degenerated vector fields.

\subsection{Examples}
\label{subsec:examples}

\begin{figure}[ht]
\centering
\includegraphics[width=0.5\textwidth]{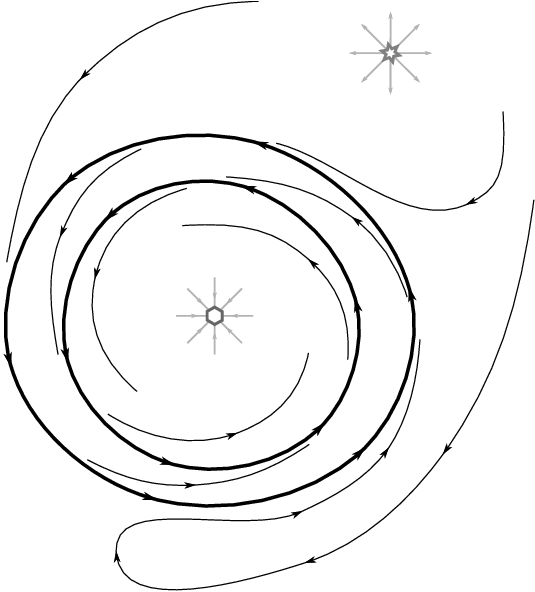}
\caption{Two parabolic cycles (thick closed orbits). There are also a~hyperbolic repeller (a star) and a hyperbolic attractor (a hexagon). The field has no other closed orbits and singular points. The $\LBS$ of any family unfolding this field consists of two thick cycles. Thinner arrows correspond to usual orbits and demonstrate that cycles are semi-stable.}\label{f1}
\end{figure}

\begin{example}\label{ex1}
    Consider a vector field shown on Fig.~\ref{f1}. The large bifurcation support of an arbitrary unfolding of this field consists of two non-hyperbolic cycles. It is disconnected. We discuss its relation with Cartesian products of bifurcations below.
\end{example}

\begin{figure}[ht]
\centering
\includegraphics[width=0.5\textwidth]{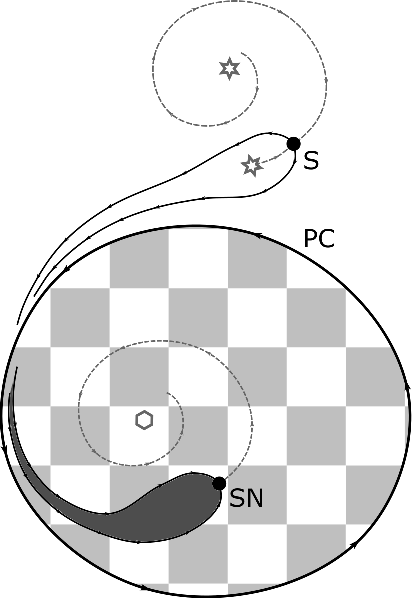}
\caption{A parabolic cycle (PC) and a saddle-node (SN). There are also two hyperbolic repellers (stars), a hyperbolic attractor (a hexagon) and a~hyperbolic saddle (S) with separatrices. The field has no other closed orbits and singular points. A generic unfolding $V$ of this field has the $\LBS$ that consists of the saddle S with two unstable separatrices, the~parabolic cycle PC and the saddle-node SN with the parabolic sector (colored with dark gray), while the $\Acc{(V)}$ contains also the whole interior domain of the parabolic cycle PC (shown by the chess filling), two stable separatrices of the saddle S and two hyperbolic repellers (stars).}\label{f2}
\end{figure}

\begin{example}\label{ex2}
See Fig.~\ref{f2}. The generic unfolding of this vector field is studied in \citep{snsn}. This example shows that the~$\LBS$ may contain an open subset of the~phase space, as well as that it may not contain the whole $\Acc$ set of~a~family.
\end{example}

A thorough analysis of those examples is not presented here for lack of space. It is straightforward and relies only on the definitions introduced in the previous subsection.

\subsection{Cartesian products of bifurcations}
\label{subsec:cartesian}

Cartesian product of bifurcations is different from the Cartesian product of families. The first one corresponds to a Cartesian product of the bases but to one and the same phase space. The second one is a family whose total space is a product of the total spaces of the factors.

First, we need the notion of a split family.

\begin{definition}[Split family]\label{def:split}
    A splitting data $(V,\ B_1,\ B_2,\ \varphi_1,\ \varphi_2)$ is a~glocal family $V = \{v_\varepsilon \left|\ \varepsilon \in B \right.\}$ with the base $B = B_1 \times B_2$ and two smooth cut functions $\varphi_1,\ \varphi_2 : S^2 \to \left[0,\ 1\right]$ with disjoint supports $\supp \varphi_j =: U_j$, $U_1 \cap U_2 = \varnothing$.
    
    Let $\varepsilon_j$ be a chart on $B_j$, $\varepsilon := (\varepsilon_1,\ \varepsilon_2)$. Denote by $p_j : B \to B$ a projection on $B_j$ (viewed as a subspace of $B$) along the other factor. The split family (or splitting, for brevity) $W = \{ w_\varepsilon \left|\ \varepsilon \in B \right.\}$ corresponding to the splitting data $(V,\ B_1,\ B_2,\ \varphi_1,\ \varphi_2)$ is given by the formula:
    $$
        w_\varepsilon = v_0 + (v_{p_1(\varepsilon)} - v_0)\ \varphi_1 + (v_{p_2(\varepsilon)} - v_0)\ \varphi_2.
    $$
\end{definition}

\begin{remark}
    Note that $w_\varepsilon = v_0$ on $S^2 \setminus (U_1 \cup U_2)$, while $w_\varepsilon = v_{p_j(\varepsilon)}$ on $\varphi^{-1}_j (1)$: the phase portrait of a vector field $w_\varepsilon$ of a split family coincides with the~phase portrait of $v_{p_1(\varepsilon)}$ in $\varphi^{-1}_1 (1)$, and of $v_{p_2(\varepsilon)}$ in $\varphi^{-1}_2 (1)$. 
\end{remark}

Split families play the role of standard Cartesian products.

\begin{definition}[Cartesian product of bifurcations]\label{def:cart}
    A glocal family $V$ with the base $B = B_1 \times B_2$ is called a Cartesian product of bifurcations in two disjoint open subsets $U_1$ and $U_2$ of $S^2$, if there exist two cut functions $\varphi_1$, $\varphi_2$ such that $\supp \varphi_j \subset U_j$ and $V$ is weakly equivalent (on $S^2$) to a split family corresponding to the splitting data $(V,\ B_1,\ B_2,\ \varphi_1,\ \varphi_2)$.
\end{definition}

\subsection{Extension triviality property}
\label{subsec:etp}

\begin{definition}
    A glocal extension of a $U$-glocal family $V_U = \{v_{\varepsilon_0} \left|\ \varepsilon_0 \in B_0 \right.\}$ is a glocal family $W = \{w_{(\varepsilon_0,\ \varepsilon_1)} \left|(\varepsilon_0,\ \varepsilon_1) \in B_0 \times B_1\right.\}$ such that $w_{(\varepsilon_0,\ 0)} = v_{\varepsilon_0}$ on $U$. A glocal extension $W$ is trivial, if a stronger condition is satisfied: $w_{(\varepsilon_0,\ \varepsilon_1)} = v_{\varepsilon_0}$ on $U$.
\end{definition}

\begin{definition}[Extension triviality]
    Let $U \subset S^2$ be an open domain. A $U$-glocal family $V_U$ with the total space $B_0 \times U$ has the extension triviality (ET) property if any glocal extension $W = \{w_{(\varepsilon_0,\ \varepsilon_1)} \left|(\varepsilon_0,\ \varepsilon_1) \in B_0 \times B_1\right.\}$ of $V_U$ is moderate equivalent in some neighborhood $M$ of $\LBS(W) \cap U$ to a trivial glocal extension $\widetilde{W} = \{\widetilde{w}_{(\varepsilon_0,\ \varepsilon_1)} \left|(\varepsilon_0,\ \varepsilon_1) \in B_0 \times B_1\right.\}$ of $V_U$, where $\widetilde{w}_{(\varepsilon_0,\ \varepsilon_1)} = v_{\varepsilon_0}$ (the second parameter, $\varepsilon_1$, is auxiliary), and the moderate equivalence
    $$
        \mathbf{H} : B \times M \to B \times M, \quad \mathbf{H}(\varepsilon,\ x) = \left(h(\varepsilon),\ H_\varepsilon (x)\right), \quad B = B_0 \times B_1,
    $$
    satisfies: $H_0 = \mathrm{id}$.
    
    A family $V_U$ has the strong extension triviality (SET) property if in the~previous definition the homeomorphism of bases $h$ may be chosen to preserve the~auxiliary~parameter~$\varepsilon_1$: $h(\varepsilon_0,\ \varepsilon_1) = (\widetilde{h}(\varepsilon_0,\ \varepsilon_1),\ \varepsilon_1)$.
\end{definition}

\begin{example}\label{ex3}
    As will be shown in \citep{Androsov}, a generic $U$-glocal family $V_U$ ($U \subset S^2$ is an open domain) that unfolds a generic degeneracy of codimension $1$ has the~SET property.
\end{example}

\begin{conjecture}
    Generic glocal families with the ET property have the SET property.
\end{conjecture}

Extension of a glocal family is an incorporation of extra parameters, of extra perturbations, into the family. A glocal family possesses the SET property, if this procedure does not affect bifurcations in the family in any significant way.

\begin{conjecture}
    Structurally stable glocal families of vector fields have the~SET property.
\end{conjecture}

\subsection{Topologically distinguished subfamilies}\label{subsec:td}

\begin{definition}
    A glocal subfamily $V \subset W$ consisting of structurally unstable orbitally topologically equivalent vector fields is topologically distinguished in~$W$ if any structurally unstable vector field of the family $W$ that is orbitally topologically equivalent to a field from the family $V$ belongs to $V$.

    Let $U \subset S^2$ be an open domain, $V$ and $W$ be the same as before. The~subfamily $V$ is topologically distinguished on $U$ in $W$ if any structurally unstable vector field from $W$ that is orbitally topologically equivalent on $U$ to a field from $V|_U$ belongs to~$V$.
\end{definition}

\begin{example}\label{ex4}
    Consider a generic two-parameter glocal family with one saddle-node fixed point and one parabolic cycle. Then the subfamily with the~parabolic cycle is topologically distinguished. It is also topologically distinguished in $U$ for any open neighborhood $U$ of the parabolic cycle.
\end{example}

\begin{example}\label{ex5}
    Consider a generic two-parameter unfolding of the field from Example~\ref{ex1}. Then a subfamily preserving one of the parabolic cycles is topologically distinguished in $U$ for any open neighborhood $U$ of the preserved parabolic cycle that does not intersect with another cycle. Since there are two parabolic cycles, there are two such topologically distinguished subfamilies.
\end{example}

\subsection{Main theorem}
\label{subsec:main}

\begin{theorem}[]\label{main}
    Let a glocal family $V = \{v_{(\varepsilon_1,\ \varepsilon_2)} \left|\ (\varepsilon_1,\ \varepsilon_2) \in (B_1 \times B_2,\ 0) \right.\}$ have a disconnected $\LBS$: $\LBS(V) = Z_1 \sqcup Z_2 =: Z$. Suppose there are disjoint open neighborhoods $U_1 \supset Z_1$ and $U_2 \supset Z_2$ such that the following two conditions are satisfied:
    \begin{itemize}
        \item[SET:] $V_{1} := V|_{B_1 \times U_1}$ has the SET property as a $U_1$-glocal family, $V_{2} := V|_{B_2 \times U_2}$ has the SET property as a $U_2$-glocal family;
        \item[TD:] a family $\{v_\varepsilon \mid \varepsilon_j = 0\}$ is topologically distinguished on $U_j$, $j = 1, 2$.
    \end{itemize}
    
    Then $V$ is a Cartesian product of bifurcations in subfamilies $V_1$ and $V_2$.
\end{theorem}

\begin{example}\label{ex6}
    Recall Example~\ref{ex1} again. A generic glocal family $V$ unfolding the field shown on Fig.~\ref{f1} has a disconnected $\LBS$. One can see that Examples~\ref{ex3}~and~\ref{ex5} suggest that SET and TD conditions are satisfied, so the Main theorem~\ref{main} is applicable. An accurate elaboration of this idea in more general setting constitutes the proof of theorem~\ref{main}. It implies that $V$ is equivalent to a Cartesian product of bifurcations of two parabolic cycles.
\end{example}

\begin{example}\label{ex7}
Consider the same vector field (denote it by $v_0$) as in Example~\ref{ex1}. Now let two-parameter family $V = \{\left.v_{(\varepsilon,\ \delta)}\ \right| (\varepsilon,\ \delta) \in (\mathbb{R} \times \mathbb{R},\ 0)\}$ be an~atypical unfolding of $v_0$ such that both parabolic cycles vanish when $\varepsilon > 0$ (so each parabolic cycle is split into two hyperbolic cycles when $\varepsilon < 0$), while $\delta$ just does nothing, i.e. $v_{(\varepsilon,\ \delta)} \equiv v_{(\varepsilon,\ 0)}$. It turns out that $\mathrm{LBS}(V)$ is a union of two parabolic cycles anyway, but their bifurcations are not independent! This family is not a Cartesian product of bifurcations. One cannot apply the Main theorem~\ref{main} because both SET and TD conditions are violated: it is impossible to split the parameter base into two factors with the properties prescribed by the conditions of theorem~\ref{main}.
\end{example}

In the last subsection we give another version of the Main theorem where we use Banach manifolds. This version is easier to apply than the previous~one.

\subsection{Banach submanifolds}
\label{subsec:banach}

\begin{definition}[\citep{lang2001fundamentals}]
    Let $\mathfrak{X}$ be a Banach space (or manifold of some fixed smoothness class). A subset $\mathfrak{M} \subset \mathfrak{X}$ is a Banach submanifold of a finite codimension $n$ if for any point $v \in \mathfrak{M}$ there is a smooth map $\psi$ of some neighborhood $\mathfrak{U}$ of $v$ in $\mathfrak{X}$ to $\mathbb{R}^n$ such that $\psi(v) = 0$, $\psi^{-1}(0) = \mathfrak{M} \cap \mathfrak{U}$ and $\psi$ has maximal rank at $v$.
\end{definition}

Given a Banach submanifold $\mathfrak{M}$ in $\Vect^{*}_k\left(S^2\right)$\footnote{$\Vect^{*}_k\left(S^2\right)$ is an open subset of $\Vect_k\left(S^2\right)$, so it is a Banach manifold.}, somitimes it is possible to consider a~Banach submanifold $\widetilde{\mathfrak{M}}$ in $\Vect^{*}_k\left(\Omega\right)$, $\Omega \subset S^2$, which in a sense characterizes $\mathfrak{M}$.

\begin{definition}[Localization of a submanifold]
    Let $\mathfrak{M}$ be a Banach submanifold in $\Vect^{*}_k\left(S^2\right)$, and $\Omega \subset S^2$ be an open subset. Let $v \in \mathfrak{M}$, and $\pi_\Omega : \Vect^{*}_k\left(S^2\right) \to \Vect^{*}_k\left(\Omega\right)$ be the natural embedding: $\pi_\Omega(v) = v|_\Omega$.  

    Suppose there exists some neighborhood $\widetilde{\mathfrak{U}} \subset \Vect^{*}_k\left(\Omega\right)$ of $v|_\Omega$ and a~smooth map $\widetilde{\psi} : \widetilde{\mathfrak{U}} \to \mathbb{R}^n$ such that $\psi := \widetilde{\psi} \circ \pi_\Omega$ satisfies three conditions:
    \begin{itemize}
        \item[1)] $\psi(v) = 0$;
        \item[2)] $\psi^{-1}(0) \cap \mathfrak{U} = \mathfrak{M} \cap \mathfrak{U}$ for some open neighborhood $\mathfrak{U}$ of $v$ in $\Vect^{*}_k\left(S^2\right)$;
        \item[3)] $\psi$ has maximal rank at $v$.
    \end{itemize}
    Then a Banach submanifold $\widetilde{\mathfrak{M}}$ defined by $\widetilde{\psi}$ in some neighborhood of $v|_\Omega$ is called a~localization of $\mathfrak{M}$ to $\Omega$ at $v$. If $\mathfrak{M}$ admits a localization to $\Omega$ at any $v \in \mathfrak{M}$, it is called $\Omega$-localizable.
\end{definition}

\begin{remark}
    It is clear that $\widetilde{\psi}$ in the definition above has the maximal rank at~$v|_\Omega$, so $\widetilde{\mathfrak{M}}$ has the same codimension as $\mathfrak{M}$.
\end{remark}

\begin{remark}
    Properties of vector fields are often formulated in terms of equalities and inequalities. More than that, such equalities and inequalities often depend on the behavior of a field on some small subsets of the phase space. That's why it is not rare when Banach manifolds concerned with bifurcations are localizable.
\end{remark}

\begin{definition}
    Let $\mathfrak{M} \subset \Vect^{*}_k\left(U\right)$, where $U \subset S^2$ is an open domain, be a Banach submanifold of finite codimension $n$. We say that $\mathfrak{M}$ has the~SET property, if any $U$-glocal $n$-parameter family $V_U$ unfolding a vector field $v \in \mathfrak{M}$ has the SET property provided that it is transversal to $\widetilde{\mathfrak{M}}$. 
\end{definition}

\begin{definition}
    We call a structurally unstable vector field $v \in \Vect^{*}_k\left(M\right)$ tame, if all structurally unstable vector fields, that are orbitally topologically equivalent to $v$ and close to $v$ in $\Vect^{*}_k\left(M\right)$, form a Banach submanifold of finite codimension. If a field $v$ is tame, the respective Banach submanifold is called a tamer of $v$.
\end{definition}

\begin{conjecture}
    A vector field $v \in \Vect^{*}_k\left(\Omega\right)$, where $\Omega$ is some open subset of $S^2$ or $S^2$ itself, that contains only a generic degeneracy of codimension $1$, is tame.
\end{conjecture}

We are ready to formulate the second version of the Main theorem now.

\begin{theorem}\label{main+}
Let $\mathfrak{M} \ni v_0$ be a Banach submanifold of $\Vect^{*}_k\left(S^2\right)$. Suppose that $\mathfrak{M} = \mathfrak{M}_1 \cap \mathfrak{M}_2$ for a pair of transversal Banach submanifolds of finite codimension. Let $V$ be an unfolding of $v_0$ transversal to $\mathfrak{M}$. Denote by $Z := \LBS (V)$. Suppose that it is a union of two connected components: $Z = Z_1 \sqcup Z_2$.

Let there exist disjoint neighborhoods $U_1$ and $U_2$ such that for each $j = 1,2$ the following holds:
\begin{itemize}
    \item[i] $Z_j \subset U_j$;
    \item[ii] $\mathfrak{M}_j$ is $U_j$-localizable;
    \item[iii] there is a localization $\widetilde{\mathfrak{M}}_j$ of $\mathfrak{M}_j$ to $U_j$ at $v_0$ with the SET-property;
    \item [iv] $\widetilde{\mathfrak{M}}_j$ is also a tamer of $v_0|_{U_j}$ in $\Vect^{*}_k\left(U_j\right)$.
\end{itemize}
Then $V$ is weakly equivalent on $S^2$ to a Cartesian product of two bifurcations, one on $U_1$, another on $U_2$.
\end{theorem}

\section{Proof of the Main theorem modulo some auxiliary facts}
\label{sec:proof-main}

We prove first the Main theorem~\ref{main}, then~\ref{main+}. The proofs use some auxiliary facts that we state now.

\subsection{Stabilization of families}
\label{subsec:stabilization}

\begin{definition}[Stabilization of a family]\label{def:stab}
    Let $V = \left\{v_{\varepsilon}\left|\ \varepsilon \in \left(B,\ 0\right)\right.\right\}$ be a~glocal family of vector fields, $Z = \LBS(V)$.

    Given an open neighborhood $U \supset Z$ and a cut function $\varphi \in C^\infty_0(S^2)$ such that $\varphi|_{U} \equiv 1$, we define the respective stabilization of $V$ as a glocal family $W = \left\{w_{\varepsilon}\left|\ \varepsilon \in \left(B,\ 0\right)\right.\right\}$, where
    $$
        w_\varepsilon = v_0 + \varphi \cdot \left(v_{\varepsilon} - v_0\right).
    $$
\end{definition}

Apparently, a well-chosen stabilization preserves the $\LBS$.

\begin{theorem}[First stabilization theorem]\label{ST1}
    For any glocal family of vector fields $V = \{v_\varepsilon \left| \ \varepsilon \in (B,\ 0) \right.\}$ and any open neighborhood $\Omega$ of $Z = \LBS (V)$ there are open neighborhoods $Z \subset \Omega_1 \ssubset \Omega_s \subset \Omega$ such that for any cut function $\varphi$ which is supported on $\overline{\Omega_s}$ and identically $1$ on $\Omega_1$ the respective stabilization $\stab V$ of $V$ has the same $\LBS$ as $V$: $\LBS (\stab V) = Z$.
\end{theorem}

This theorem has an important analogue concerning splittings.

\begin{theorem}[Second stabilization theorem]\label{ST2}
    For any glocal family $V$ and open neighborhood $\Omega$ of $Z = \LBS (V)$ there are open neighborhoods $Z \subset \Omega_1 \ssubset \Omega_s \subset \Omega$ such that given any two cut functions $\varphi_1$ and $\varphi_2$ with properties:
    \begin{itemize}
        \item $\varphi_1$ and $\varphi_2$ have disjoint supports,
        \item $\supp{\varphi_1} \sqcup \supp{\varphi_2} \subset \overline{\Omega_s}$,
        \item $(\varphi_1 + \varphi_2)|_{\Omega_1} \equiv 1$,
    \end{itemize}
    the inclusion $\LBS (\splt V) \subset Z$ holds for the splitting $\splt V$ of $V$ that corresponds to a splitting data $(V,\ B_1,\ B_2,\ \varphi_1,\ \varphi_2)$, see Def.~\ref{def:split}.
\end{theorem}

\subsection{Splitting lemma}
\label{subsec:splitting}

Recall that split families are introduced above (see Def.~\ref{def:split}). Here we construct a moderate equivalence between a family satisfying certain conditions and the corresponding split family.

\begin{lemma}\label{splitting-of-two}[Splitting lemma]
    Let $V = \left\{v_{(\varepsilon_1,\ \varepsilon_2)} \left|\ (\varepsilon_1,\ \varepsilon_2) \in \left(B_1 \times B_2,\ 0\right)\right.\right\}$ be an $M$-glocal family, $\LBS(V) = Z_1 \sqcup Z_2$.
    Suppose there are disjoint open neighborhoods $U_1 \supset Z_1$ and $U_2 \supset Z_2$ such that the following two conditions are satisfied:
    \begin{itemize}
        \item[SET:] $V_{1} := V|_{B_1 \times U_1}$ has the SET property as a $U_1$-glocal family, $V_{2} := V|_{B_2 \times U_2}$ has the SET property as a $U_2$-glocal family;
        \item[TD:] a family $\{v_\varepsilon \mid \varepsilon_j = 0\}$ is topologically distinguished on $U_j$, $j = 1, 2$.
    \end{itemize}

    Then $V$ is moderate equivalent on some open neighborhood $U$, $Z_1 \sqcup Z_2 \subset U \subset U_1 \sqcup U_2$, to the split family $W$ corresponding to splitting data $(V,\ B_1,\ B_2,\ \varphi_1,\ \varphi_2 )$, provided that $\varphi_1|_{U_1} \equiv 1$ and $\varphi_2|_{U_2} \equiv 1$. Moreover, the~moderate equivalence map for $\varepsilon = 0$ is identity.
\end{lemma}

This statement admits a straightforward generalization to any finite number of parameters and neighborhoods. However, we omit in this paper the proof of the multiparameter version, because it goes by induction, while the proof of the lemma shows both the base and the induction step quite comprehensively.

\begin{proof}
    Notice that $V$ is a glocal extension of $V_1$. Because of the SET property that $V_{1}$ has, $V$ is moderate equivalent to a~trivial glocal extension of $V_1$ on some open set $U_1'$, $Z_1 \subset U_1' \subset U_1$. Let us denote by $(h,\ H)$ the corresponding map. Define a map $\mathbf{H}$ by its restrictions:
    \begin{align}
            \mathbf{H} = (h,\ H) \quad \text{on} \ B \times U_1', \nonumber\\
            \mathbf{H} = (h,\ \mathrm{id}) \quad \text{on} \ B \times U_2, \nonumber
    \end{align}
    where $B = B_1 \times B_2$.
    
    We see that the map constructed is a moderate equivalence on $U_1' \sqcup U_2$ between $V$ and $W' = \{w'_{(\varepsilon_1,\ \varepsilon_2)} \mid (\varepsilon_1,\ \varepsilon_2) \in \left(B_1 \times B_2,\ 0\right)\}$, where
    $$
        w'_{(\varepsilon_1,\ \varepsilon_2)} := v_0 + \varphi_1 \cdot (v_{(\varepsilon_1,\ 0)} - v_0) + \varphi_2 \cdot (v_{h^{-1}(\varepsilon_1,\ \varepsilon_2)} - v_0).
    $$
    By the definition of the SET property the equivalence preserves the auxiliary parameter. In other words, $h$ preserves the fibers $\varepsilon_2 = \text{const}$.
    
    Structurally unstable vector fields on $U_1'$ have to be mapped to structurally unstable vector fields on $U_1'$, so $h$ has to send the coordinate submanifold $\varepsilon_1 = 0$ to itself due to the TD condition of the lemma. Altogether this implies that $h(0,\ \varepsilon_2) = (0,\ \varepsilon_2)$ or, equivalently, $h^{-1}(0,\ \varepsilon_2) = (0,\ \varepsilon_2)$. Hence $w'_{(0,\ \varepsilon_2)} = v_{(0,\ \varepsilon_2)}$ on $U_2$, so $W'$ is a glocal extension of $V_2$.
    
    Now we can repeat the last few lines: since $V_2$ has the SET property, $W'$ is moderate equivalent to a trivial glocal extension of $V_2$ on some open set $U_2'$, $Z_2 \subset U_2' \subset U_2$. Let us denote $(g,\ G)$ the corresponding map. Define a map $\mathbf{G}$ by its restrictions:
    \begin{align}
            \mathbf{G} = (g,\ \mathrm{id}) \quad \text{on} \ B \times U_1', \nonumber\\
            \mathbf{G} = (g,\ G) \quad \text{on} \ B \times U_2'. \nonumber
    \end{align}
    We see that this map is a moderate equivalence on $U := U_1' \sqcup U_2'$ between $W'$ and $W''= \{w''_{(\varepsilon_1,\ \varepsilon_2)} \mid (\varepsilon_1,\ \varepsilon_2) \in \left(B_1 \times B_2,\ 0\right)\}$, where
    $$
        w''_{(\varepsilon_1,\ \varepsilon_2)} := v_0 + \varphi_1 \cdot (w_{g^{-1}(\varepsilon_1,\ 0)} - v_0) + \varphi_2 \cdot (v_{(0,\ \varepsilon_2)} - v_0).
    $$
    Analogously to $h$, $g$ preserves the fibers $\varepsilon_1 = \text{const}$ and maps the coordinate submanifold $\varepsilon_2 = 0$ to itself, so $g(\varepsilon_1,\ 0) = (\varepsilon_1,\ 0)$. Hence $W'' = W$, and we are done: $\mathbf{G} \circ \mathbf{H}$ is the moderate equivalence we are looking for.

    Both $\mathbf{H}$ and $\mathbf{G}$ are equal to the identity map on $0 \times U_1' \sqcup U_2'$. Hence the~same holds for $\mathbf{G} \circ \mathbf{H}$.
\end{proof}

\subsection{Proof of the Main theorem \ref{main}}
\label{subsec:proof-main}

Both Main theorems easily follow from theorem \ref{ST2} and the Splitting lemma.

    \begin{proof}
    Choose two smooth cut-functions $\{\widetilde{\varphi_1},\ \widetilde{\varphi_2}\}$ such that $\widetilde{\varphi_j}|_{U_j} \equiv 1$, $j = 1,\ 2$. Then $(V,\ B_1,\ B_2,\ \widetilde{\varphi_1},\ \widetilde{\varphi_2})$ is a splitting data for $V$. Denote by $\widetilde{W}$ the respective splitting of $V$.

    Notice that conditions SET and TD of the Main theorem~\ref{main} and of the~Splitting lemma are exactly the same.
    The Splitting lemma \ref{splitting-of-two} implies that $V$ is moderate equivalent to $\widetilde{W}$ on some open neighborhood $U$, $Z_1 \sqcup Z_2 \subset U \subset U_1 \sqcup U_2$, and this moderate equivalence, say $\mathbf{H}$, is identity on $0 \times U$. More explicitly, we have $\mathbf{H}(\varepsilon,\ x) = \left(h(\varepsilon),\ H_\varepsilon(x)\right)$, and $H_0 = \mathrm{id}$.

    Take open neighborhoods $Z \subset \Omega_1 \ssubset \Omega_s \subset U$ provided by theorem \ref{ST2}. Choose two smooth cut-functions $\{\varphi_1,\ \varphi_2\}$ such that $\supp{\varphi_j} \subset \overline{\Omega_s} \cap U_j$ and $\varphi_j|_{\Omega_1 \cap U_j} \equiv 1$, $j = 1,\ 2$. This gives us a splitting data $(V,\ B_1,\ B_2,\ \varphi_1,\ \varphi_2)$ and the respective splitting $\splt V$ of $V$. By~the~Second stabilization theorem, $\LBS (\splt V) \subset \LBS (V)$.    
    
    Hence $\Omega_1$ is an open neighborhood of $\LBS (\splt V)$. Let $\mathbf{G}$ be the restriction of $\mathbf{H}$ to $B \times \Omega_1$, where $B = B_1 \times B_2$. Then $\mathbf{G}$ is a moderate equivalence between $V$~and~$\splt V$ in a neighborhood of their large bifurcation supports, because $\splt V$~and~$\widetilde{W}$ coincide in $\Omega_1$. Since $\mathbf{H}$ agrees with an orbital topological equivalence between $v_0$ and $w_0$ ($v_0 = w_0$, so this is just the identity map in our case) when $(\varepsilon_1,\ \varepsilon_2) = 0$, the same holds for $\mathbf{G}$. Therefore, $V$ is weakly equivalent to $\splt V$ on $S^2$ by theorem \ref{LBS-main-theorem}. This concludes the proof as this is how the Cartesian product of bifurcations is defined.
    \end{proof}

\subsection{Proof of the Main theorem \ref{main+}}\label{subsec:proof-main+}

\begin{proof}
    Recall that the conditions $(ii-iii)$ of the theorem ensure that there exists a localization of $\mathfrak{M}_j$ to $U_j$, $\widetilde{\mathfrak{M}}_j$, with the SET property.
    
    As $V$ is transversal to $\mathfrak{M}$, it is transversal to each $\mathfrak{M}_j$. Hence, without loss of generality we may assume that the pullback of each $\mathfrak{M}_j \cap V$ to the base of $V$ is a coordinate subspace $\varepsilon_j = 0$, while the base of $V$ is $B_V = \left(\mathbb{R}^{n_1} \times \mathbb{R}^{n_2},\ 0\right)$. Then $V_1 := \left\{v_{(\varepsilon_1,\ 0)} \left|\ \varepsilon_1 \in (\mathbb{R}^{n_1},\ 0)\right.\right\}$ is transversal to $\mathfrak{M}_1$, and $V_2 := \left\{v_{(0,\ \varepsilon_2)} \left|\ \varepsilon_2 \in (\mathbb{R}^{n_2},\ 0)\right.\right\}$ is transversal to $\mathfrak{M}_2$. Hence $\widetilde{V}_1 := V_1|_{U_1}$ is transversal to $\widetilde{\mathfrak{M}}_1$, $\widetilde{V}_2 := V_2|_{U_2}$ is transversal to $\widetilde{\mathfrak{M}}_2$, and so they have the SET property as $U_1$-glocal and $U_2$-glocal families respectively.

    By the condition $(iv)$, $\widetilde{\mathfrak{M}}_j$ is a tamer of $v_0|_{U_j}$ in $\Vect^{*}_k\left(U_j\right)$. It means that $\widetilde{\mathfrak{M}}_j$ contains all structurally unstable vector fields that are orbitally topologically equivalent to $v_0|_{U_j}$ and close to $v_0|_{U_j}$ in $\Vect^{*}_k\left(U_j\right)$. Therefore, $V_1$ is topologically distinguished on $U_2$, $V_2$ is topologically distinguished on~$U_1$.

    Finally, the condition $(i)$ gives us inclusions $Z_j \subset U_j$. Since $Z_j$ are components of $\LBS(V)$, we can apply the Main theorem~\ref{main} to $V$, $V_1$ and $V_2$ (with the TD condition satisfied by $V_2$ and $V_1$ respectively) to deduce that $V$ is indeed a Cartesian product of bifurcations.
\end{proof}

\subsection{Bifurcation diagrams}
\label{subsec:diagrams}

If a glocal family $V = \{v_{(\varepsilon_1,\ \varepsilon_2)} \left|\ (\varepsilon_1,\ \varepsilon_2) \in (B_1 \times B_2,\ 0) \right.\}$ satisfies the~conditions of the Main theorem~\ref{main}, it is a Cartesian product of bifurcations. According to definition~\ref{def:cart} there exist two cut functions $\varphi_1$, $\varphi_2$ such that $\supp \varphi_j \subset U_j$ and $V$ is weakly equivalent on $S^2$ to a split family $W$ corresponding to the splitting data $(V,\ B_1,\ B_2,\ \varphi_1,\ \varphi_2)$. The following theorem tells us more about $W$:
\begin{theorem}\label{th:bd}
    The bifurcation diagram ($\BD$) of the split family $W$ is
    $$
        \BD{(W)} = \BD{(V_1)} \times B_2 \cup B_1 \times \BD{(V_2)},
    $$
    where $V_j$ is the subfamily of $V$ corresponding to $B_j$:
    $$
        \begin{aligned}
            V_1 &= \{v_{(\varepsilon_1,\ 0)} \left|\ \varepsilon_1 \in B_1 \right.\}, \\
            V_2 &= \{v_{(0,\ \varepsilon_2)} \left|\ \varepsilon_2 \in B_2 \right.\}.
        \end{aligned}
    $$
\end{theorem}

\begin{example}\label{bd}
    \begin{figure}[ht]
        \centering
        \includegraphics[width=0.5\textwidth]{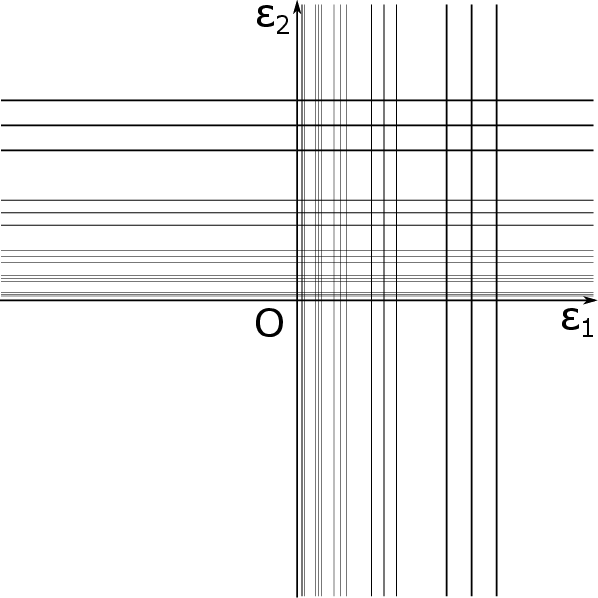}
        \caption{An example of the bifurcation diagram that $W$ might have.}\label{f3}
    \end{figure}

    Look at Fig.~\ref{f3}. This bifurcation diagram of a split family $W$ corresponds to the case when both $V_1$ and $V_2$ are responsible for occurrence of sparkling saddle connections\footnote{\sparkle}, but in~two different domains of $S^2$. The~diagram consists of two independent stacks of lines: the vertical one of $V_1$ ($\BD{(V_1)} \times B_2$) and the horizontal one of $V_2$ ($B_1 \times \BD{(V_2)}$).
\end{example}

\begin{proof}
    Let us fix $\varepsilon = (\varepsilon_1,\ \varepsilon_2) \notin \BD{(V_1)} \times B_2 \cup B_1 \times \BD{(V_2)}$ and prove that $w_{(\varepsilon_1,\ \varepsilon_2)}$ is structurally stable.

    By the condition of the Main theorem~\ref{main}, $\LBS{(V)} = Z_1 \sqcup Z_2$. Recall that $W = \{w_\varepsilon\}$, where
    $$
        w_\varepsilon = v_0 + (v_{(\varepsilon_1,\ 0)} - v_0)\ \varphi_1 + (v_{(0,\ \varepsilon_2)} - v_0)\ \varphi_2,
    $$
    and cut-functions $\varphi_1$, $\varphi_2$ have disjoint supports and are equal to $1$ on $U_1 \supset Z_1$ and $U_2 \supset Z_2$ respectively. Denote by $U := U_1 \sqcup U_2$. It follows from the formula, that $w_{(\varepsilon_1,\ \varepsilon_2)}$ coincides with $v_{(\varepsilon_1,\ 0)}$ on $U_1$ and with $v_{(0,\ \varepsilon_2)}$ on $U_2$. Hence, the restriction of $w_{(\varepsilon_1,\ \varepsilon_2)}$ to $U$ satisfies the Andronov--Pontryagin cryterion or, equivalently, $w_{(\varepsilon_1,\ \varepsilon_2)}$ has no non-Andronov elements (see the~definition in the next section) in $U$.

    By Theorem~\ref{ST2}, $\LBS{(W)} \subset \LBS{(V)} \subset U$. Due to that fact, we can choose a small neighborhood $U_0$ of $\LBS{(W)}$ which satisfies Lemma~\ref{NML}, Lemma~\ref{NEL} and is contained in $U$. All non-hyperbolic cycles and singular points of $W$ (for small values of $\varepsilon$, at~least) are contained in $U_0$ by lemma~\ref{NML}. Analogously, all saddle connections of $w_{(\varepsilon_1,\ \varepsilon_2)}$ can only occur in $U_0$ by lemma~\ref{NEL}. Therefore, $w_{(\varepsilon_1,\ \varepsilon_2)}$ has no no-Andronov elements, and thus is structurally stable.
\end{proof}

\begin{remark}
    The Main theorem~\ref{main} states that $V$ is weakly equivalent to $W$ on~$S^2$. In particular, $\BD{(V)}$ is homeomorphic to $\BD{(W)}$.
\end{remark}

\section{Cartesian products of bifurcations in two-parameter families}
\label{sec:two-param}

\subsection{Non-Andronov elements and disconnected LBS}\label{subsec:non-andronov}

In contrast to the Andronov--Pontryagin criterion, characterizing structurally stable vector fields on $S^2$, let us call a non-Andronov element any of the following orbits:
\begin{itemize}
    \item a non-hyperbolic singular point,
    \item a non-hyperbolic limit cycle,
    \item a common separatrix of two singular points, no matter whether they are hyperbolic or not.
\end{itemize}

\begin{proposition}\label{prop:9}
    Any connected component of the $\LBS$ contains at least one non-Andronov element.
\end{proposition}

\begin{proof}
    Let $v_0$ be a field unfolded by a glocal family $V$. Take an orbit $\gamma$ of $v_0$ from a connected component of $\LBS(V)$. Let $A_\gamma$ and $W_\gamma$ be the $\alpha$- and $\omega$-limit sets of $\gamma$ respectively. By Poincare--Bendixon theorem, only singular points, cycles or polycycles can be limit sets of $\gamma$. Note that $A_\gamma$, $W_\gamma$ and $\gamma$ lie in the same connected component of $\LBS(V)$.

    It is proved in~\citep{Goncharuk-LBS} that $\LBS(V)$ contains no hyperbolic attractors and repellers, and no hyperbolic limit cycles. Hence, either one of the limit sets $A_\gamma$ and $W_\gamma$ is a non-hyperbolic singular point, limit cycle or polycycle, and, thus, contains a non-Andronov element, or they both are hyperbolic saddles. In~this case, $\gamma$ is a saddle connection, so it is a non-Andronov element itself.
\end{proof}

\subsection{The basic list}
\label{subsec:basic-list}

It is well-known that only six classes of degeneracies can be encountered unavoidably in generic one-parameter families:

\begin{theorem}[\citep{Sotomayor1974}]\label{th:sotomayor}
In generic one-parameter families on $S^2$ only degeneracies from the following {\bf basic list} are met:

  \begin{itemize}
    \item[$\AH$:]
    A non-hyperbolic singular point with a pair of non-zero pure imaginary eigenvalues such that the first Lyapunov focus value is non-zero;
    \item[$\SN$:]
    A saddle-node singular point of multiplicity exactly two;
    \item[$\HC$:]
    A homoclinic curve of a saddle-node singular point of multiplicity exactly two that enters the saddle node through the interior of the parabolic sector;
    \item[$\SC$:]
    A saddle connection between two different hyperbolic saddles;
    \item[$\SL$:]
    A separatrix loop of a hyperbolic saddle with the characteristic number not equal to $1$;
    \item[$\PC$:]
    A parabolic cycle of multiplicity two.
  \end{itemize}
  
\end{theorem}

\subsection{Degeneracies related to the disconnected LBS in two-parameter families}
\label{subsec:two-param}

The following results will be proved in the forthcoming papers by Androsov, Bakiev, Filimonov and Ilyashenko, but on the heuristic level they can be easily explained right here.

\begin{conjecture}\label{th:ABFI}
    In a generic two-parameter glocal family of vector fields with a disconnected $\LBS$ exactly two degeneracies from the basic list occur; each one corresponds to a connected component of the $\LBS$. Moreover, a~so-called ''non-synchronization`` condition holds for each degeneracy of the $\PC$ class.
\end{conjecture}

\begin{proof}[Heuristic proof]
    By Proposition~\ref{prop:9} any connected component of $\LBS(V)$ contains at least one non-Andronov element. The codimension of the corresponding degeneracy is one or greater. Suppose that at least one of these degeneracies is outside the basic list. Then its codimension is greater than one, and the total codimension of the degeneracy in $V$ is greater than $2$. But a degeneracy of codimension greater than $2$ cannot occur in generic two-parameter families.
\end{proof}

This is not a rigorous proof because the last statement ''a degeneracy of codimension greater than $2$ cannot occur in generic two-parameter families`` is neither strictly stated nor proved. A rigorous proof of conjecture~\ref{th:ABFI} is a~subject of future work.

\begin{corollary}
    Let $V$ be a generic two-parameter glocal family of vector fields with a disconnected $\LBS$ $Z = Z_1 \sqcup Z_2$, $U_1 \supset Z_1$ and $U_2 \supset Z_2$ are disjoint open neighborhoods. Then for each $j = 1,\ 2$:
    \begin{itemize}
        \item the restriction of $V$ to $U_j$ is an extension of some generic one-parameter $U_j$-glocal family;
        \item $U_j$ contains only one non-Andronov element corresponding to a degeneracy from the basic list, and the ''non-synchronization`` condition holds for each degeneracy of the $\PC$ class.
    \end{itemize}
\end{corollary}

\begin{theorem}[\citep{Androsov}]\label{th:androsov}
    Generic one-parameter $U$-glocal families of vector fields, where $U$ is a domain on $S^2$, have the SET property.
\end{theorem}

Conjectures~\ref{th:ABFI}~and~\ref{th:androsov} allow us to apply the Main theorem~\ref{main} to generic two-parameter families with a disconnected~$\LBS$, and thus to prove conjecture~\ref{th2}.

\section{Proof of the stabilization theorems}\label{sec:proof-stab}

\subsection{Properties of LBS}
\label{subsec:properties}

We say that a subset $Z$ of $S^2$ admits an arbitrary small open neighborhood with some property, if any open neighborhood $\Omega$ of $Z$ contains a strictly smaller open neighborhood $\Omega' \ssubset \Omega$ of $Z$ with this property.

The next four lemmas are based on the results from \citep{Goncharuk-LBS}. They are a bit reformulated, because in this paper a different notation is adopted. We will use those lemmas to get some nice neighborhoods of $\LBS (V)$ and deduce the~stabilization theorems.

\begin{lemma}[No cycles of mixed location, proposition 4.11 from \citep{Goncharuk-LBS}]\label{NML}
    For any $n$-parameter glocal family $V$ of vector fields there exists an arbitrary small open neighborhood $\Omega$ of $\LBS (V)$ for which one can find an open neighborhood of zero $B \subset \mathbb{R}^n$ s.t. for all $\varepsilon \in B$
    \begin{itemize}
        \item each singular point of $v_{\varepsilon}$ is either inside $\Omega$, or belongs to a continuous family $P_{\varepsilon}$, $\varepsilon \in (B_p, 0)$, of hyperbolic singular points of $v_{\varepsilon}$ such that $P_0 \notin L B S(V)$;
        \item each limit cycle of $v_{\varepsilon}$ is either inside $\Omega$, or belongs to a continuous family $c_{\varepsilon}$, $\varepsilon \in (B_c, 0)$, of hyperbolic limit cycles of $v_{\varepsilon}$ such that $c_0$ does not belong to $L B S(V)$.
    \end{itemize}
\end{lemma}

\begin{definition}[Definition 4.6 from \citep{Goncharuk-LBS}, corrected]
    $$\LBS^* (V) = \LBS (V) \setminus \{\text{non-interesting cycles and singular points of } v_0\}.$$
\end{definition}

There is a technical reason why we need $\LBS^*$ -- it is much easier to formulate next lemmas for $\LBS^*$. Let us state without any explanation that some degenerate limit cycles can be non-interesting even though they bifurcate. As for degenerate equilibria, we consider non-interesting only Andronov--Hopf singular points.

An intuition behind such a division is the following. Non-interesting cycles and singular points, when they bifurcate, do not interact with separatrices winding around them. In other words, we don't consider them interesting limit sets. This is a manual exclusion which does not let us treat them the~same way as all other components of $\LBS$. The real purpose of this exclusion is, of~course, not to make our lives harder, but to shrink $\LBS$ as much as possible.

\begin{lemma}[No-entrance lemma, a combination of definition 4.9 and lemma 4.10 from \citep{Goncharuk-LBS}]\label{NEL}
    For any glocal family $V = \{\left. v_\varepsilon \right|\ \varepsilon \in (\mathbb{R}^n,\ 0)\}$, $\LBS^* (V)$ admits an arbitrary small open neighborhood $\Omega$ with the no-entrance property: there exists an open neighborhood of zero $B \subset \mathbb{R}^n$ s.t. for any $\varepsilon \in B$ no unstable separatrix of $v_\varepsilon$ enters $\Omega$, and no stable separatrix of $v_\varepsilon$ quits $\Omega$ when time increases. Or, equivalently, for any $\varepsilon \in B$ each separatrix of $v_\varepsilon$ either does not intersect $\partial \Omega$, or emanates either in forward or backward time from a~singular point inside $\Omega$ and intersects $\partial \Omega$ at a single point.
\end{lemma}

\subsection{Interesting and non-interesting parts of the LBS. Choice of neighborhoods}
\label{subsec:choice}

    Let us, first of all, choose $\hat{\Omega} \subset \Omega$ applying lemma \ref{NML} to $\Omega$ and $V$.

    Let $V = \{v_\varepsilon \mid \varepsilon \in B\}$ be the same as in assumption of either one of the stabilization theorems. Let $W = \{w_\varepsilon \mid \varepsilon \in B \}$  be either $\stab V$ or $\splt V$. We already know that $W$ is glocal. As $v_0 = w_0$, we have: $ \ELBS(v_0) = \ELBS(w_0)$. So according to definition \ref{def:cLBS} of the $\LBS$ we only need to prove that
    $$
    \ELBS (v_0) \cap \operatorname{Acc}{(W)} = \ELBS (v_0) \cap \operatorname{Acc}{(V)},
    $$
    for some choice of nested open neighborhoods of $\LBS(V)$. Denote by $Z^*_V \ (Z^0_V)$ the interesting (non-interesting) part of $Z = \LBS(V)$.
   
    Denote by $\Omega^0 \subset \hat{\Omega}$ a neighborhood of $Z^0_V$ bounded by curves without contact, and $\Omega_0 \ssubset \Omega^0$ -- a similar but smaller neighborhood of $Z^0_V$.

    Denote by $\Omega^* \ssubset \hat{\Omega}$ a neighborhood of $Z^*_V$ that satisfies the No-entrance lemmas. Let $\Omega_* \ssubset \Omega^*$ be a similar neighborhood.
    
    \begin{remark}\label{NML-rem}
        Observe that the margin $\overline{\Omega^*} \setminus \Omega_*$ (and, in particular, $\partial \Omega^*$) contains no singular points and intersects no limit cycles of $v_0$ because we have already used lemma \ref{NML}. This holds also for $v_{\varepsilon}$ when $\varepsilon$ is small enough.
    \end{remark}
    
    Let $\Omega_1$ be the union $\Omega_0 \cup \Omega_*$ and $\Omega_s$ (index $s$ for support) be the union $\Omega^0 \cup \Omega^*$. Recall that the cut function $\varphi$ used to define $W$ have the following properties: $\supp \varphi \subset \overline{\Omega_s}$ and $\varphi \equiv 1$ on $\Omega_1$. Note that $V \equiv W$ in $\Omega_1$, and $w_\varepsilon \equiv w_0 = v_0$ outside $\Omega_s$. We use $\Omega_1 \ssubset \Omega_s$ to formulate the stabilization theorems, but we do not refer to those neighborhoods in the proof below. Instead, we operate with $\Omega_0 \ssubset \Omega^0$ and $\Omega_* \ssubset \Omega^*$.

    The relation
    $
    Z^0_V = Z^0_W
    $
    is trivial and does not even depend on how~$W$ is constructed (on the choice of neighborhoods) as long as $v_0 = w_0$. By~the~definition, all non-interesting non-hyperbolic limit cycles and Andronov--Hopf singular points of $v_0 = w_0$ belong to both $\LBS (V)$ and $\LBS (W)$.

    Let us now prove the relation
    $
    Z^*_V = Z^*_W
    $
    or, equivalently, that
    $$
    (\Acc(W) \setminus \Omega_0) \cap \ELBS (v_0) = (\Acc(V) \setminus \Omega_0) \cap \ELBS (v_0).
    $$
    Surprisingly, the inclusion $Z^*_V \subset Z^*_W $ is much easier to prove than the inverse one. In what follows, we do not write that we choose a convergent subsequence from a sequence we consider, but rather assume that the sequence is convergent itself.

\subsection{Proof of the simple inclusion: the LBS of the original family belongs to the LBS of a stabilization}
\label{subsec:simple}

In this subsection $W = \{w_\varepsilon \mid \varepsilon \in B\} = \stab V$.

\begin{proof}
By~definition \ref{def:cLBS}, it is sufficient to prove that
$$
\Acc{(\stab V)} \cap \ELBS(v_0) \supset \Acc{(V)} \cap \ELBS(v_0).
$$

\proofpart{1}{$\overline{\operatorname{Per} V} \cap\{\varepsilon=0\} \cap \ELBS (v_0) \subset \overline{\operatorname{Per} \stab V} \cap\{\varepsilon=0\} \cap \ELBS (v_0)$}

Let $p \in \overline{\operatorname{Per} V} \cap\{\varepsilon=0\}$. Then there exists a sequence of points $\{p_n\}$ such that $p_n$ lies on a limit cycle $l_{n}$ of $v_{\varepsilon_n}$, $\lim_{n \to \infty} \varepsilon_n = 0$ and $\lim_{n \to \infty} p_n = p$. According to Lemma~\ref{NML}, when $n$ is big enough, each $l_n$ belongs entirely either to $S^2 \setminus (\Omega^* \cup \Omega^0)$ or to $\Omega_* \cup \Omega_0$.

\begin{itemize}
    \item[Case 1:] Each $l_n$ belongs entirely to $S^2 \setminus (\Omega^* \cup \Omega^0)$ when $n$ is big enough. Hence $p$ necessarily lies on a~hyperbolic limit cycle of $v_0 = w_0$ by lemma~\ref{NML}. Therefore, $p \in \operatorname{Per} (w_0) \subset \operatorname{Per} \stab V$.

    \item[Case 2:] Each $l_n$ belongs entirely to $\Omega_* \cup \Omega_0$ when $n$ is big enough. Recall that $\stab V$ coincides with $V$ on $\Omega_* \cup \Omega_0$, so in this case the inclusion $p \in \overline{\operatorname{Per} \stab V}$ is trivial.
\end{itemize}

We proved that $p \in \overline{\operatorname{Per} \stab V} \cap\{\varepsilon=0\}$. In other words, there is an~inclusion $\overline{\operatorname{Per} V} \cap\{\varepsilon=0\} \subset \overline{\operatorname{Per} \stab V} \cap\{\varepsilon=0\}$. The stated inclusion, $\overline{\operatorname{Per} V} \cap\{\varepsilon=0\} \cap \ELBS (v_0) \subset \overline{\operatorname{Per} \stab V} \cap\{\varepsilon=0\} \cap \ELBS (v_0)$, now follows immediately.

\proofpart{2}{$\overline{\operatorname{Sep} V} \cap\{\varepsilon=0\} \cap \ELBS (v_0) \subset \overline{\operatorname{Sep} \stab V} \cap\{\varepsilon=0\} \cap \ELBS (v_0)$}

Let $p \in \overline{\operatorname{Sep} V} \cap\{\varepsilon=0\}\cap \ELBS (v_0)$. Then there exists a sequence of points $\{p_n\}$ such that each $p_n$ lies on a separatrix $\gamma_n$ of $v_{\varepsilon_n}$, $\lim_{n \to \infty} \varepsilon_n = 0$ and $\lim_{n \to \infty} p_n = p$. Let each separatrix $\gamma_n$ be unstable and emanate from a~singular point $s_n$ of $v_{\varepsilon_n}$. According to the agreement above, we may assume that $\lim_{n \to \infty} s_n = s$ for some singular point $s$ of $v_0 = w_0$.

Since $p \in \LBS (V)$, $p \in \Omega_* \cup \Omega_0$. Two cases are possible depending on where $s$ is located.

\begin{itemize}
    \item[Case 1:] The singular point $s$ belongs to $\Omega^*$. Recall that there is no singular points of $v_0$ in the margin $\Omega^* \setminus \Omega_*$ by remark~\ref{NML-rem}, hence $s$ is, actually, in $\Omega_*$. Passing to a subsequence, we may assume that each $s_n$ belongs to $\Omega_*$. Thus, $\Omega_*$ contains almost all terms of sequences $p_n$ and $s_n$.

    Now it's time to use the fact that $v_\varepsilon|_{\Omega_*} = w_\varepsilon|_{\Omega_*}$. In particular, one can see that $s_n$ is a singular point of $w_{\varepsilon_n}$, and this singular point of $w_{\varepsilon_n}$ has an unstable separatrix which surely coincides with $\gamma_n$ in $\Omega_*$ until it leaves~$\Omega_*$.

    Since $\Omega_*$ has the no-entrance property, when $n$ is big enough, the~whole arc of $\gamma_n$ between $s_n$ and $p_n$ is contained in $\Omega_*$. Hence, $p_n$ does belong to the~corresponding separatrix of~$w_{\varepsilon_n}$, and we can finally conclude that $p \in \overline{\operatorname{Sep} \stab V} \cap \{\varepsilon=0\} \cap \ELBS (v_0)$.

    \item[Case 2:] Singular point $s$ does not belong to $\Omega^*$. According to remark~\ref{NML-rem}, $s \notin \partial \Omega^*$. So, passing to a subsequence, we may assume that each $s_n$ belongs to $S^2 \setminus \Omega^*$. Recall that $\gamma_n$ emanates from $s_n$. The No-entrance lemma implies that $\gamma_n$ belongs entirely to $S^2 \setminus \Omega^*$ for almost all $n$. Hence, $p_n \in \Omega_0$ for almost all $n$. But $\Omega_0 \cap \LBS (V)$ contains only cycles and Andronov--Hopf singular points, so in this case the necessary inclusion $p \in \overline{\operatorname{Sep} \stab V} \cap \{\varepsilon=0\} \cap \ELBS (v_0)$ is trivial.
\end{itemize}
\end{proof}

\subsection{Proof of the hard inclusion: the LBS of a stabilization belongs to the LBS of the original family, proof of Theorem~\ref{ST2}}
\label{subsec:hard}

\begin{proof}
In this section, we will complete the proof of Theorem~\ref{ST1} and prove Theorem~\ref{ST2}. In what follows, $W = \{w_\varepsilon \mid \varepsilon \in B\}$ is either $\stab V$ (no~assumptions on $Z = \LBS (V)$), or $\splt V$ ($Z$ is disconnected).

The proof of part 1 below is simple, although there are four short cases. The proof of the part 2 is divided into $5$ claims and the final step. The claims for both theorems coincide. We use mainly the fact that if $\varepsilon_n \to 0$, then $w_{\varepsilon_n} \to v_0$ (the convergence is uniform). The final step is slightly different for the two theorems. At this stage we have to refer to the definition of $W$. We also refer to it implicitly in the fourth claim, where we need a~fact that is valid for both the stabilization and the splitting.

\proofpart{1}{$\overline{\operatorname{Per} V} \cap \{\varepsilon=0\} \cap \ELBS (v_0) \supset \overline{\operatorname{Per} W} \cap\{\varepsilon=0\} \cap \ELBS (v_0)$}

Let $l_n$ be a sequence of limit cycles in $\operatorname{Per}W$: $l_n$ is a limit cycle of~$w_{\varepsilon_n}$ and $\lim_{n \to \infty} \varepsilon_n = 0$. Then $l_n$ accumulate to one of the following four sets:
\begin{itemize}
    \item[(i)] a hyperbolic limit cycle of~$w_0 = v_0$,
    \item[(ii)] a non-hyperbolic limit cycle of~$w_0 = v_0$,
    \item[(iii)] a non-hyperbolic singular point of~$w_0 = v_0$,
    \item[(iv)] a polycycle of~$w_0 = v_0$.
\end{itemize}

The first two cases, (i) and (ii), are trivial, because each limit cycle of $v_0$ belongs to $\operatorname{Per}(v_0) \subset \operatorname{Per} V$ by the definition.

Consider the last two cases, (iii) and (iv). Since the~limit $l = \lim_{n \to \infty} l_n$ is either a polycycle or a non-hyperbolic singular point, it belongs to $Z = \LBS (V)$. But then there is an open neighborhood of $Z$ on $S^2$ that contains~$l$ and where $W$ and $V$ coincide. Hence $l_n$ is a trajectory of $v_{\varepsilon_n}$ for all but a~finite number of indexes. As $\lim_{n \to \infty} \varepsilon_n = 0$, $l \subset \overline{\operatorname{Per} V} \cap \{\varepsilon=0\}$.

We proved that $\overline{\operatorname{Per} V} \cap\{\varepsilon=0\} \supset \overline{\operatorname{Per} W} \cap\{\varepsilon=0\}$. The stated inclusion, $\overline{\operatorname{Per} V} \cap \{\varepsilon=0\} \cap \ELBS (v_0) \supset \overline{\operatorname{Per} W} \cap\{\varepsilon=0\} \cap \ELBS (v_0)$, now follows immediately.

\proofpart{2}{$\overline{\operatorname{Sep} W} \cap\{\varepsilon=0\} \cap \ELBS (v_0) \subset \overline{\operatorname{Sep} V} \cap\{\varepsilon=0\} \cap \ELBS (v_0)$}

We begin in a similar way to the proof of the inverse inclusion. Let 
$$
p \in \overline {\operatorname{Sep} W} \cap \{\varepsilon = 0\} \cap \ELBS (v_0).
$$
Then $p = \lim_{n \to \infty} p_n, \ p_n \in \gamma_n$, where $\gamma_n$ is a separatrix of $w_n := w_{\varepsilon_n}$, $\varepsilon_n \to 0$. Let $\gamma_n$ be an unstable separatrix of a singular point $s_n$, and $ s= \lim_{n \to \infty} s_n$ is a~singular point of $w_0 = v_0$.

\begin{claim}\label{cl1}
    $s \in \Omega_*$.
\end{claim}

\begin{proof}
    Suppose first that $s$ lies outside of $\overline{\Omega^* \cup \Omega^0}$. Then the same holds for almost all of $s_n$. The~vector field $w_n$ coincides with $v_0 = w_0$ outside of $\Omega^* \cup \Omega^0$. Hence, $s_n = s$ and each separatrix $\gamma_n$ belongs entirely to $S^2 \setminus \Omega^*$ by the No-entrance lemma. There are two different cases depending on either the~orbits $\gamma_n$ enter $\Omega^0$ or not.

    \begin{itemize}
        \item[Case 1:] All but finitely many $\gamma_n$ do not enter $\Omega^0$. Recall again that the~vector field $w_n$ coincides with $v_0 = w_0$ outside of $\Omega^* \cup \Omega^0$. Hence $p$ belongs to $\overline{\gamma}$ for some separatrix $\gamma$ of $v_0$ emerging from $s$ and disjoint from $\Omega^* \cup \Omega^0$. But the $\omega$-limit set of $\gamma$ is non-interesting (otherwise $\gamma$ would be a part of $\LBS (V)$), so $p \not \in \ELBS (v_0)$, a~contradiction.

        \item[Case 2:] All but finitely many $\gamma_n$ enter $\Omega^0$. By the construction, $\Omega^0$ is a disjoint union of annular neighborhoods of limit cycles. Each such neighborhood is a disjoint union of a limit cycle and two open annuli; each annulus is either an absorbing or repelling domain\footnote{Those domains are absorbing or repelling with respect to $v_0 = w_0$, but these properties are preserved by continuity for small values of parameters.}. Hence $p$ belongs to a~trajectory of $v_0$ such that its $\alpha$- or $\omega$-limit set is contained in $\Omega^0$. (Note that $p$ may or may not be located on a separatrix of $s$ or on a limit sets of some separatrix of $s$, because some nested non-interesting limit cycles of even multiplicity inside $\Omega^0$ can vanish due to bifurcations.) But~$\Omega^0$ contains only non-interesting limit sets, so this limit set is either an~Andronov--Hopf singular point or a non-hyperbolic non-interesting limit cycle. Thus, $p \not \in \ELBS (v_0)$, a~contradiction again.
    \end{itemize}

    We have just proved that $s \in \overline{\Omega^* \cup \Omega^0}$. By remark~\ref{NML-rem}, $\partial \Omega^*$ contains no singular points of $v_0$. The same also true for $\partial \Omega^0$ by construction. Observe that $\Omega^0$ contains only monodromic singular points, so $s \notin \Omega^0$. This way we conclude that $s$ lies inside $\Omega^*$. Hence $s \in \Omega_*$, because the margin $\Omega^* \setminus \Omega_*$ contains no singular points of~$v_0$ by remark~\ref{NML-rem}.
\end{proof}

Denote by $\gamma$ the $v_0$-trajectory of $p$, and by $\alpha(\gamma)$ the $\alpha$-limit set of $\gamma$.

\begin{claim}\label{cl2}
     $\alpha(\gamma) \subset \Omega_*$.
\end{claim}

\begin{proof}
    Indeed, if $\alpha(\gamma) \subset S^2 \setminus \Omega^*$, then $\alpha(\gamma)$ is either a non-interesting limit cycle or a non-interesting singular point, so $p \notin \ELBS(v_0)$, a~contradiction. Hence $\alpha(\gamma) \subset \Omega^*$, and the claim follows since the margin $\Omega^* \setminus \Omega_*$ contains no singular points or limit cycles of $w_0 = v_0$ by remark~\ref{NML-rem}.
\end{proof}

The next claim is the most conceptual among all five. Heuristically it says that connected components of $\Omega_*$ are isolated in some sense from each other. Or, in other words, an orbit that connects two disjoint components of $\Omega_*$ and has interesting $\alpha$- and $\omega$-limits sets does not exist.

\begin{claim}\label{cl3}
    If both limit sets of $\gamma$ are interesting, then $\gamma \subset \LBS(V)$.
\end{claim}

\begin{proof}
    If both limit sets of $\gamma$ are interesting, then $\gamma \subset \ELBS(w_0) = \ELBS(v_0)$. Let $\gamma_*$ be the arc of $\gamma$ such that $\gamma_* \subset \Omega_*$, $\gamma_*$ is backward invariant and $\alpha(\gamma) \subset \overline{\gamma_*}$. Due to claim~\ref{cl2}, such an arc exists. Moreover, it is clearly unique.
    
    Recall that $\gamma \subset \lim_{n \to \infty} \gamma_n$, and each $\gamma_n$ is a~separatrix~of~$w_{\varepsilon_n}$ emanating from singular point $s_n$. According to claim~\ref{cl1}, we may assume that each $s_n$ is in $\Omega_*$. Observe that $A \subset \lim_{n \to \infty} \gamma_n$. The crucial point is to prove that for any open neighborhood of $\gamma_*$ for all sufficiently large values of $n$ $\gamma_n$ does not leave $\Omega_*$ before it enters a small neighborhood of $\gamma_*$.

    By the construction, $\gamma_* \subset \Omega_*$, so it's enough to consider an open neighborhood $U_* \subset \Omega_*$. For~the~connected component $U$ of~$\Omega_*$ which contains $s$ one can find a~sequence~$\varepsilon_n'$, $\lim_{n \to \infty} \varepsilon_n' = 0$, such that $w_{\varepsilon_n}|_{U} = v_{\varepsilon_n'}|_{U}$\footnote{This statement is true regardless of whether $W$ is a stabilization or a splitting. The proof is a bit different, though, but uses only the respective definitions in both cases.}. In~particular, $v_{\varepsilon_n'}$ has a~separatrix $\gamma_n'$ that coincides with $\gamma_n$ until both leave $U \subset \Omega_*$ for the first time. But if $\gamma_n$ left $\Omega_*$ and $n$ is big enough, then $\gamma_n$ would enter $\Omega_*$ again because $U_* \subset \Omega_*$. This is prohibited by the No-entrance lemma for all sufficiently large values of $n$.

    Hence $s$ and $\gamma_*$ are contained in the same connected component $U$ of $\Omega_*$, and we know that $v_{\varepsilon_n'}$ has a~separatrix $\gamma_n'$ that coincides with $\gamma_n$ until both leave $U \subset \Omega_*$ for the first time, which happens before $\gamma_n$ enters an arbitrarily small neighborhood of $\gamma_*$ unless $n$ is not sufficiently large. This way we have showed that $\gamma_* \subset \overline{\operatorname{Sep} V} \cap\{\varepsilon=0\}$.
    
    Therefore, by definition~\ref{def:cLBS} of $\LBS$, $\gamma_* \subset \LBS(V)$. But $\LBS(V)$ is $v_0$-invariant, so the claim follows.
\end{proof}

\begin{claim}\label{cl4}
    $p \in \Omega_*$.
\end{claim}

\begin{proof}
    Let, on the contrary, $p \notin \Omega_*$.

    Suppose first, that at least one limit set of $\gamma$ under $v_0$ is non-interesting. Then it follows immediately that $p \notin \ELBS(v_0)$. This is a contradiction with the choice of $p$.

    Thus, both limit sets of $\gamma$ under $v_0$ are interesting. It follows from claim~\ref{cl2} that $\gamma \subset \LBS(V) \subset \Omega_* \cup \Omega_0$. Since $\Omega_0$ is disjoint from $\Omega_*$, and $\gamma$ is connected, it follows that $p \in \gamma \subset \Omega_*$. That contradicts our initial assumption.
\end{proof}

\begin{claim}\label{cl5}
    Arcs $\bar \gamma_n(s_n, p_n)$ of separatrices of vector fields $w_n$ belong to $\Omega_*$
\end{claim}

\begin{proof}
    As $\Omega_*$ is open, we get from claim~\ref{cl1} that $s_n \in \Omega_*$, and we get from claim~\ref{cl4} that $p_n \in \Omega_*$. By contraposition, suppose these arcs, $\gamma_n(s_n, p_n)$, intersect the boundary $\partial \Omega_*$ for the first time at the points $r_n$.
    
    Consider $\omega_{v_0}(r)$, where $r = \lim_{n \to \infty} r_n$. It is disjoint from $\Omega^*$, since the~trajectory of $r$ cannot enter $\Omega^*$ again by~the~No-entrance lemma. Then $\omega_{v_0}(r)$ is a non-interesting limit set. Hence the~forward trajectory of $r$ enters $Y \subset S^2 \setminus \Omega^*$, an absorbing domain for $v_0 = w_0$. But this domain remains absorbing for $w_n$ as well (for big enough~$n$), because $w_n = w_0$ on $S^2 \setminus (\Omega^* \cup \Omega^0)$, and $\partial \Omega^0$ consists of transversal loops.
    
    It follows by the continuous dependence on parameters theorem that the~orbits of $r_n$ enter $Y$ and, thereby, do not return to $\Omega_*$. This contradicts the~fact that $p_n \in \Omega_*$. Thus, the~arcs~$\bar \gamma_n(s_n, p_n)$ belong to $\Omega_*$.
\end{proof}

\textit{End of the proof.} At this point, the proofs of Theorems \ref{ST1} and \ref{ST2} diverge, but the conclusion of claim~\ref{cl5} is essential for both of them.

\begin{proof}[Proof of Theorem \ref{ST1}: $W = \stab V$]
    The families $V$ and $\stab V$ coincide in~$\Omega_*$. Hence, the arcs $\bar \gamma_n(s_n, p_n)$ of separatrices of $w_{\varepsilon_n}$ are at~the~same time arcs of separatrices of $v_{\varepsilon_n}$. As $\varepsilon_n \to 0$, and $p_n \to p$, we conclude that 
    $
        p \in \overline {\operatorname{Sep} V} \cap \{\varepsilon = 0\} \cap \ELBS (v_0),
    $
    as required.
\end{proof}

\begin{proof}[Proof of Theorem \ref{ST2}: $W = \splt V$.]
    Then $\Omega_*$ splits in two neighborhoods of $Z_1$ and $Z_2$; denote them by $\Omega_*^1$ and $\Omega_*^2$. We may assume that all the arcs $\bar \gamma_n(s_n, p_n)$ belong to $\Omega_*^1$, where the vector field $w_{\varepsilon_n}$ for $\varepsilon_n = (\varepsilon_n^1, \varepsilon_n^2)$ has the~form $v_{\varepsilon_n^1, 0}$. So the arcs $\bar \gamma_n$ of separatrices of the vector fields $w_{\varepsilon_n}$ are at the~same time arcs of separatrices of the vector fields $v_{\varepsilon_n^1, 0}$. As $\varepsilon_n \to 0$, and $p_n \to p$, we conclude again that
    $
        p \in \overline {\operatorname{Sep} V} \cap \{\varepsilon = 0\} \cap \ELBS (v_0).
    $
\end{proof}

Theorems \ref{ST1} and \ref{ST2} are completely proved.
\end{proof}

\section*{Acknowledgments}

The authors are grateful to the anonymous reviewer for his/her thorough and constructive feedback, which significantly contributed to improving the clarity and rigor of this paper. We also thank our colleagues: Ilya Androsov for many friendly discussions and inspiring ideas, Andrey Dukov for his valuable comments and suggestions, and Dmitry Filimonov for his constant support, inexhaustible interest in this work, crucial remarks and insightful observations.

\section*{Funding}

This article is an output of a research project implemented as part of the Basic Research Program
at HSE University (№ 075-00648-25-00).

\bibliographystyle{unsrtnat}
\bibliography{references}

\end{document}